\documentclass[a4paper,12pt]{amsart}
\usepackage[utf8]{inputenc}
\usepackage{amsmath, amssymb, amsthm}
\usepackage{geometry}
\usepackage{enumitem}
\usepackage{hyperref}

\usepackage{listings}
\usepackage{xcolor}

\lstdefinestyle{python}{
  language=Python,
  basicstyle=\ttfamily\small,
  keywordstyle=\color{blue},
  commentstyle=\color{gray},
  stringstyle=\color{red!70!black},
  numbers=left,
  numberstyle=\tiny\color{gray},
  showstringspaces=false,
  breaklines=true,
  frame=single,
  tabsize=2,
  captionpos=b
}

\newtheorem{theorem}{Theorem}[section]
\newtheorem{lemma}[theorem]{Lemma}

\title[Tribonacci palindromic concatenations of repdigits]{\textbf{Tribonacci numbers that are palindromic concatenations \\
of two distinct repdigits}}

\author[K. Faye]{Khady Faye}
\address{K. Faye, ,UFR SAT, Universit\'e Alioune Diop, Bambey, 30, S\'en\'egal}
\email{fkhady94@gmail.com}

\author[B. Faye]{Bernadette Faye}
\address{B. Faye,
	UFR SAT, Universit\'e Alioune Diop, Bambey, 30, S\'en\'egal}
\email{bernadette.faye@uadb.edu.sn}

\begin{document}
\maketitle

\begin{abstract}
 Let $(T_n)_{n\ge1}$ be the Tribonacci sequence, defined by $T_1=T_2=1$,
$T_3=2$, and $T_n=T_{n-1}+T_{n-2}+T_{n-3}$ for $n\ge4$. In this paper, we study the Tribonacci numbers whose decimal representation is a palindromic
concatenation of two distinct repdigits, that is, numbers of the shape $\overline{d_1\cdots d_1\,d_2\cdots d_2\,d_1\cdots d_1}$ with
$1\le d_1\le9$, $0\le d_2\le9$, $d_1\ne d_2$. We prove that no Tribonacci number is a palindromic
concatenation of two distinct repdigits.
The proof combines lower bounds for linear forms in logarithms, the Baker-Davenport reduction
method, and a finite computational verification.\\



\end{abstract}

\section{Introduction}

A \emph{repdigit} is a positive integer all of whose decimal digits are
equal, meaning  of the form $$N:=d\cdot\frac{10^j-1}{9}, $$ for some
digit $1\le d\le9$ and length $j\ge1$. A positive integer is a
\emph{palindromic} concatenation of two distinct repdigits if it has the
decimal representation
\begin{equation}\label{eq:shape}
\overline{\underbrace{d_1\cdots d_1}_{\ell}\ \underbrace{d_2\cdots d_2}_{m}\ \underbrace{d_1\cdots d_1}_{\ell}},
\qquad 1\le d_1\le9,\ 0\le d_2\le9,\ d_1\ne d_2,\ \ell,m\ge1.
\end{equation}
Over the last few years, the problem of determining all members of a
given linear recurrence sequence with this decimal shape has been solved
for a number of classical sequences. For example, in
\cite{Batte2024} Batte and al. show that
no Lucas number has this shape. In \cite{ChalebgwaDdamulira2021} Chalebgwa and Ddamulira proved that the padovan sequence has exactly two solutions of this shape, namely
$P_{15}=151$ and $P_{21}=616$.  The $k$-generalized Fibonacci
numbers and the  $k$-generalized Lucas numbers for
$k\ge3$  have been studied in ~\cite{BatteLuca2025}  and  ~\cite{BatteKaggwa2025} respectively. There are no solutions again. Recently  Ddamulira proved in \cite{Ddamulira2025} that the Tribonacci-Lucas numbers is also empty of solutions.\\
The proofs in this circle of results follow a common template, going
back to the treatment of related repdigit problems for the Tribonacci
sequence itself in~\cite{Ddamulira2020} (concatenations of \emph{two}
plain repdigits, a different combinatorial shape from \eqref{eq:shape}). The main argument of the proof is to
write the recurrence term via its Binet formula, translate the digit
condition into a small linear form in logarithms of algebraic numbers,
bound it from below by Matveev's theorem~\cite{Matveev2000}, and reduce
the resulting astronomically large bound to a range that allows an
exhaustive computer search via the Baker--Davenport
method~\cite{BakerDavenport1969} or LLL lattice
reduction~\cite{Cohen1993}.\\

Despite this active line of work, and despite the fact that the
Tribonacci sequence's own repdigit-concatenation problem was already
treated in~\cite{Ddamulira2020}, the specific question we addressed here
is the palindromic concatenations of two \emph{distinct} repdigits, in the
sense of relation \eqref{eq:shape}, for the Tribonacci sequence $(T_n)$ itself. We need to say that the Tribonacci sequence is different to its companion Tribonacci-Lucas sequence.\\
We prove the following result:
\begin{theorem}\label{thm:main}
There is no Tribonacci number that is a palindromic concatenation of
two distinct repdigits.
\end{theorem}

The paper is organized as follows. Section~\ref{sec:setup}
gives all the preliminaries that are related to the tribonacci sequence, the  needed lower bound for linear forms in logs and the Baker-Davenport reduction algorithm. In  Section~\ref{sec:lemma1}, we prove our main theorem by first giving a relation
between the Tribonacci index $n$ and the total digit length $2\ell+m$. Moreover, we construct the first linear form in
logarithms attached to the equation and derives an explicit upper bound
on the block length $\ell$ in terms of $n$.  We conclude that section by
constructing a second, independent linear form bounding $\ell+m$, closing
the loop to an absolute effective bound on $n$.
Finally, we reduce the bound on $n$ by using  a
Baker--Davenport reduction to a range small enough for exhaustive
search, which will complete the proof the theorem.

\section{Preliminaries}\label{sec:setup}

\subsection{The Tribonacci sequence}
The characteristic polynomial of the Tribonacci sequence $ \{T_n\}_{n\geq 0} $ is given by
\begin{eqnarray*}
\Psi (x) = x^3-x^2-x-1.
\end{eqnarray*}
This polynomial is irreducible in $ \mathbb{Q}[x] $ and has a positive real root 
\begin{eqnarray*}
\alpha = \dfrac{1}{3}\left(1+(19+3\sqrt{33})^{1/3}+(19-3\sqrt{33})^{1/3}\right),
\end{eqnarray*}
lying strictly outside the unit circle and two complex conjugate roots $\beta$ and $ \gamma $ lying strictly inside the unit circle. Furthermore, $ |\beta|=|\gamma|=\alpha^{-1/2} $. According to Dresden and Zu \cite{Dresden}, a Binet-like formula for the $ k $-generalized Fibonacci sequences is established. For the Tribonacci sequences, it states that
\begin{eqnarray}
T_n=C_{\alpha}\alpha^{n-1}+C_{\beta}\beta^{n-1}+C_{\gamma}\gamma^{n-1} ~~~ \text{  for all  } n\geq 1,\label{tribo}
\end{eqnarray}
where $ C_X=(X-1)/(4X-6) $. Dresden and Zu \cite{Dresden}, also showed that the contribution of the complex conjugate roots $ \beta $ and $ \gamma $ to the right-hand side of \eqref{tribo} is very small. And more precisely,
\begin{eqnarray}\label{triboa}
\left|T_{n}-C_{\alpha}\alpha^{n-1}\right|<\dfrac{1}{2}~~~\text{  for all  }  n\geq 1.
\end{eqnarray}
It is also a well known fact ( see \cite{Luca16, Ddamulira2}) that 
\begin{eqnarray}\label{tribo1}
\alpha^{n-2}\leq T_n\leq \alpha^{n-1} ~~~\text{  holds for all  } n\geq 1.
\end{eqnarray}
Let $ \mathbb{K}:=\mathbb{Q}(\alpha, \beta) $ be the splitting field of the polynomial $ \Psi $ over $ \mathbb{Q} $. Then $ [\mathbb{K}, \mathbb{Q}]=6 $. Furthermore, $ [\mathbb{Q}(\alpha):\mathbb{Q}]=3 $. The Galois group of $ \mathbb{K} $ over $ \mathbb{Q} $ is given by
\begin{eqnarray*}
\mathcal{G}:=\text{Gal}(\mathbb{K/Q})\cong \{(1), (\alpha\beta),(\alpha\gamma), (\beta\gamma), (\alpha\beta\gamma), (\alpha\gamma\beta)\} \cong S_3.
\end{eqnarray*}
Thus, we identify the automorphisms of $ \mathcal{G} $ with the permutations of the roots of the polynomial $ \Psi $. For example, the permutation $ (\alpha\gamma) $ corresponds to the automorphism $ \sigma: \alpha \to \gamma, \gamma \to \alpha, \beta\to \beta $.\\

In a different setting, the Tribonacci sequence can be also seen as
\[ T_k = c_{\alpha} \alpha_T^k + c_{\beta} \beta_T^k + c_{\gamma} \gamma_T^k \quad \mbox{ for all } k \geqslant 0, \]
where $\alpha_T$, $\beta_T$ and  $\gamma_T$ are the roots of the characteristic equation $x^3-x^2 - x -1=0$, with
\begin{align*}
\alpha_T = & \ \dfrac{1}{3}\left( 1+\sqrt[3]{19+3\sqrt{33}} + \sqrt[3]{19-3\sqrt{33}} \right),\\
\beta_T = & \ \dfrac{1}{6}\left( 2- \sqrt[3]{19+3\sqrt{33}} - \sqrt[3]{19-3\sqrt{33}} \right)+\dfrac{\sqrt{3}}{6}i\left(\sqrt[3]{19+3\sqrt{33}}- \sqrt[3]{19-3\sqrt{33}} \right),\\
\gamma_T = & \ \dfrac{1}{6}\left( 2- \sqrt[3]{19+3\sqrt{33}} - \sqrt[3]{19-3\sqrt{33}} \right)-\dfrac{\sqrt{3}}{6}i\left(\sqrt[3]{19+3\sqrt{33}}- \sqrt[3]{19-3\sqrt{33}} \right),
\end{align*}
and the coefficients
\begin{align*}
 c_\alpha = &\ \dfrac{\alpha_T}{(\alpha_T-\beta_T)(\alpha_T-\gamma_T)}=\dfrac{1}{-\alpha_T^2+4\alpha_T-1},\\
 c_\beta= &\ \dfrac{\beta_T}{(\beta_T-\alpha_T)(\beta_T-\gamma_T)}=\dfrac{1}{-\beta_T^2+4\beta_T-1},\\
 c_\gamma =&\ \dfrac{\gamma_T}{(\gamma_T-\alpha_T)(\gamma_T-\beta_T)}= \dfrac{1}{-\gamma_T^2+4\gamma_T-1}
\end{align*}
are the roots of the polynomial $44x^3-2x -1$. Note that
\begin{align*}
& 1.839 < \alpha_T < 1.840 && 0.336 < c_{\alpha} < 0.337\\
& \beta_T= \overline{\gamma_T} && 0.737< |\beta_T|= |\gamma_T| < 0.738 \\
& c_{\beta} = \overline{c_{\gamma}}  && 0.259< |c_{\beta}| = |c_{\gamma}| < 0.260.
\end{align*}

Moreover,  the Binet
formula for the Tribonacci sequence  gives
\begin{equation}\label{eq:binet}
T_n=g(\alpha)\alpha^n+e(n),\qquad g(\lambda):=\frac{\lambda}{\Psi'(\lambda)},\qquad
e(n):=g(\beta)\beta^n+g(\bar\beta)\bar\beta^n.
\end{equation}

\begin{lemma}\label{lem:height}
The algebraic number $g(\alpha)=\alpha/\Psi'(\alpha)$ has minimal
polynomial $44x^3-2x-1$ over $\mathbb Q$ and absolute logarithmic Weil
height
\[
h(g(\alpha))=\frac{\log44}{3}=1.26139654\ldots.
\]
\end{lemma}

\begin{proof}
Working in the basis $\{1,\alpha,\alpha^2\}$ of $K:=\mathbb Q(\alpha)$
(where $\alpha^3=\alpha^2+\alpha+1$), the extended Euclidean algorithm
applied, in exact rational arithmetic, to $\Psi'(x)=3x^2-2x-1$ and $\Psi(x)$
gives
\[
g(\alpha)=\frac{-4-3\alpha+5\alpha^2}{22}.
\]
The matrix of multiplication by $g(\alpha)$ on $K$ in this basis has
trace $0$, sum of principal $2\times2$ minors $-1/22$, and determinant
$1/44$; its characteristic polynomial is therefore
$x^3-\tfrac1{22}x-\tfrac1{44}$, which clears to the primitive integer
polynomial $44x^3-2x-1$. By the rational root test (divisors of $1$ and
of $44$), this polynomial has no rational root, so it is indeed the
minimal polynomial of $g(\alpha)$ over $\mathbb Q$, of degree $3$. Its
three roots are $g(\alpha)\approx0.336228$ and a complex conjugate pair
of modulus $\approx0.260$, all of modulus $<1$; hence $\log^+|\cdot|=0$
at each root, and
\[
h(g(\alpha))=\frac13\Big(\log44+\sum_{i=1}^{3}\log^+|g(\alpha)_i|\Big)=\frac{\log44}{3}.\qedhere
\]
\end{proof}

\subsection{Linear forms in logarithms}
In order to prove our main result Theorem \ref{thm:main}, we need to use a Baker--type lower bound for a nonzero linear form in logarithms of algebraic numbers. There are many such  bounds in the literature  like that of Baker and W{\"u}stholz from \cite{bawu07}.  We use the one of Matveev from \cite{Matveev2000}.  
Matveev \cite{Matveev2000} proved the following theorem, which is one of our main tools in this paper.

\noindent
Let $ \gamma $ be an algebraic number of degree $ d $ with minimal primitive polynomial over the integers
$$ a_{0}x^{d}+ a_{1}x^{d-1}+\cdots+a_{d} = a_{0}\prod_{i=1}^{d}(x-\gamma^{(i)}),$$
where the leading coefficient $ a_{0} $ is positive and the $ \eta^{(i)} $'s are the conjugates of $ \gamma $. Then the \textit{logarithmic height} of $ \gamma $ is given by
$$ h(\gamma) := \dfrac{1}{d}\left( \log a_{0} + \sum_{i=1}^{d}\log\left(\max\{|\gamma^{(i)}|, 1\}\right)\right).$$

In particular, if $ \gamma = p/q $ is a rational number with $ \gcd (p,q) = 1 $ and $ q>0 $, then $ h(\gamma) = \log\max\{|p|, q\} $. The following are some of the properties of the logarithmic height function $ h(\cdot) $, which will be used in the next sections of this paper without reference:
\begin{eqnarray}
h(\eta\pm \gamma) &\leq& h(\eta) +h(\gamma) +\log 2,\nonumber\\
h(\eta\gamma^{\pm 1})&\leq & h(\eta) + h(\gamma),\\
h(\eta^{s}) &=& |s|h(\eta) ~~~~~~ (s\in\mathbb{Z}). \nonumber
\end{eqnarray}

\begin{theorem}[Matveev]\label{Matveev11} Let $\gamma_1,\ldots,\gamma_t$ be positive real algebraic numbers in a real algebraic number field 
$\mathbb{K}$ of degree $D$, $b_1,\ldots,b_t$ be nonzero integers, and assume that
\begin{equation}
\label{eq:Lambda}
\Lambda:=\gamma_1^{b_1}\cdots\gamma_t^{b_t} - 1,
\end{equation}
is nonzero. Then
$$
\log |\Lambda| > -1.4\times 30^{t+3}\times t^{4.5}\times D^{2}(1+\log D)(1+\log B)A_1\cdots A_t,
$$
where
$$
B\geq\max\{|b_1|, \ldots, |b_t|\},
$$
and
$$A
_i \geq \max\{Dh(\gamma_i), |\log\gamma_i|, 0.16\},\qquad {\text{for all}}\qquad i=1,\ldots,t.
$$
\end{theorem} 

\subsection{Baker-Davenport reduction procedure}
During the course of our calculations, we get some upper bounds on our variables which are too large, thus we need to reduce them. To do so, we use some results from the theory of continued fractions. Specifically, for a nonhomogeneous linear form in two integer variables, we use a slight variation of a result due to Dujella and Peth{\H o} \cite{DujellaPetho1998}, which  itself is a generalization of a result of Baker and Davenport \cite{BakerDavenport1969}.

\noindent
For a real number $X$, we write  $||X||:= \min\{|X-n|: n\in\mathbb{Z}\}$ for the distance from $X$ to the nearest integer.
\begin{lemma}[Dujella, Peth\H o]\label{Dujjella}
Let $M$ be a positive integer, $p/q$ be a convergent of the continued fraction of the irrational number $\tau$ such that $q>6M$, and  $A,B,\mu$ be some real numbers with $A>0$ and $B>1$. Let further 
$\varepsilon: = ||\mu q||-M||\tau q||$. If $ \varepsilon > 0 $, then there is no solution to the inequality
$$
0<|u\tau-v+\mu|<AB^{-w},
$$
in positive integers $u,v$ and $w$ with
$$ 
u\le M \quad {\text{and}}\quad w\ge \dfrac{\log(Aq/\varepsilon)}{\log B}.
$$
\end{lemma}


We will also need the following Lemma.

\begin{lemma}{\label{eqn1}}
If $x\in \mathbb{R}$ satisfies $\left|x\right|<\frac{1}{2}$, then $\left|\log(1+x)\right|<\frac{3}{2}\left|x\right|.$

\end{lemma}


%
%

\section{Proof of Theorem \ref{thm:main}}\label{sec:lemma1}
\subsection{Bounding the variable $n$, $m$ and $\ell$}
Before proving our result, we will first prove some important relations.
By relation \eqref{eq:shape}, we set  $R_j=(10^j-1)/9$ to be a repunit. Then our the Diophantine problem became
\begin{equation}\label{eq:main}
T_n=d_1R_\ell(10^{\ell+m}+1)+d_2R_m10^\ell.
\end{equation}
Multiplying by $9$ and setting $N:=2\ell+m$, \eqref{eq:main} becomes
\begin{equation}\label{eq:main9}
9T_n+d_1=d_1\cdot10^{N}+(d_2-d_1)\cdot10^{\ell+m}+(d_1-d_2)\cdot10^{\ell}.
\end{equation}

We have the following first Lemma:

\begin{lemma}\label{lem:size}
Every solution of \eqref{eq:main} with $\ell,m\ge1$ satisfies
\[
|n-3.778583\,N|<2
\quad\hbox{with}\quad N=2\ell+m,
\]
where $\frac{\log 10}{\log \alpha}<3.778583.$
\end{lemma}

\begin{proof}
By \eqref{eq:shape}, $T_n$ has exactly $N=2\ell+m$ decimal digits, so
$10^{N-1}\le T_n<10^N$, and in particular $10^{N-1}\le9T_n+d_1<10^{N+1}$. On the other hand, by
\eqref{eq:binet}, $T_n=g(\alpha)\alpha^n+e(n)$, where
$|e(n)|\le2|g(\beta)||\beta|^n<2\cdot0.26\cdot0.7374<0.4$ for $n\ge1$. Hence
$$10^{N-1}-0.4<g(\alpha)\alpha^n<10^N+0.4.$$
Since $N\ge3$, we have $10^{N-1}-0.4\ge0.996\cdot10^{N-1}$ and $10^N+0.4\le1.0004\cdot10^N$.
Taking logarithms, and using $n\log\alpha=\log\big(g(\alpha)\alpha^n\big)-\log g(\alpha)$ with
$-\log g(\alpha)<1.0899$, we get
$$-1.217<n\log\alpha-N\log10<1.091.$$
Dividing by $\log\alpha=0.6093$ gives $|n-3.778583\,N|<1.217/0.6093<2$, which is the claim.
\end{proof}


Set $K:=\mathbb Q(\alpha)$, $C:=9g(\alpha)/d_1\in K$, and
\begin{equation}\label{eq:Lambda1}
\Lambda_1:=C\,\alpha^n\,10^{-N}-1=\eta_1^{b_1}\eta_2^{b_2}\eta_3^{b_3}-1,
\end{equation}
where $\quad \eta_1=\alpha\ ,\ \eta_2=10 ,\ \ \eta_3=C, b_1=n, b_2=-N,$ and $ b_3=1.$

From \eqref{eq:main9}, dividing by $d_1\cdot10^N$ and substituting
\eqref{eq:binet}, one gets the  identity
\begin{equation}\label{eq:exact-expansion}
\Lambda_1=\frac{d_2-d_1}{d_1}\cdot10^{-\ell}+\frac{d_1-d_2}{d_1}\cdot10^{-(\ell+m)}
-\frac{9e(n)+d_1}{d_1\cdot10^{N}}.
\end{equation}
Since $\ell<\ell+m<N$, the first term on the right dominates for
$\ell,m$ large, and a direct estimate of the last two terms, using
$\ell,m\ge1$, $1\le d_1,d_2\le9$, $|e(n)|<0.4$ (indeed $|e(n)|\le2|g(\beta)||\beta|^n<0.384$ for $n\ge1$), and $N\ge2\ell+1$, gives
\begin{equation}\label{eq:lambda1bound}
0<|\Lambda_1|<12\cdot10^{-\ell}.
\end{equation}

Before applying Matveev's theorem, we need to prove that $\Lambda_1\ne0$ for every solution of \eqref{eq:main} with $\ell,m\ge1$.



Suppose $\Lambda_1=0$. Then $C\alpha^n=10^N$, i.e.
\begin{equation}\label{eq:nonzero1-hyp}
9g(\alpha)\alpha^n=d_1\cdot10^N\in\mathbb Q.
\end{equation}
The minimal polynomial $\Psi(x)=x^3-x^2-x-1$ of $\alpha$ has three roots
$\alpha,\beta,\bar\beta$, so there is a field embedding
$\sigma:K=\mathbb Q(\alpha)\hookrightarrow\mathbb C$ with
$\sigma(\alpha)=\beta$. Since $g$ is a rational function with rational
coefficients ($g(\lambda)=\lambda/\Psi'(\lambda)$), $\sigma(g(\alpha))=g(\beta)$,
and $\sigma$ fixes $\mathbb Q$ pointwise, so applying $\sigma$ to both
sides of \eqref{eq:nonzero1-hyp} gives
\[
9g(\beta)\beta^n=d_1\cdot10^N.
\]
But $$|9g(\beta)\beta^n|\le9|g(\beta)||\beta|^n\le9\cdot0.26\cdot1<2.34$$
for all $n\ge0$ using the fact that $|g(\beta)|\approx0.25999$ and $|\beta|<1$. Therefore, $d_1\cdot10^N\ge10^N\ge10^3=1000$ for $\ell,m\ge1$. This give that 
$N=2\ell+m\ge3$. This is a contradiction, so $\Lambda_1\ne0$.\\

We now apply Matveev's theorem \cite{Matveev2000} to $\Lambda_1$ with
$t=D=3$. We can choose $B=n$ since from Lemma 3.1, one note that $n>N$.  Moreover,  $A_1=\log\alpha$, $A_2=3\log10$, and  $A_3=17$ since 
$$h(C)=h(9g(\alpha)/d_1)\le\log9+h(g(\alpha))+\log d_1\le2\log9+h(g(\alpha))=5.6558$$
for every $1\le d_1\le9$ by Lemma \ref{lem:height}, so $Dh(C)\le16.97<17.$
Thus, we obtain that 
For every $n\ge1$,
\begin{equation}\label{eq:approx}
\log|\Lambda_1|\ge-1.94\times 10^{14}(1+\log n).
.\end{equation}


Combining inequality \eqref{eq:approx} with the  bound obtained in 
\eqref{eq:lambda1bound} yields the following.

\begin{lemma}\label{thm:ell-bound}
Every solution $(n,\ell,m,d_1,d_2)$ of \eqref{eq:main} with $\ell,m\ge1$
satisfies
\[
\ell<F(n):=\frac{\log12+1.94\times 10^{14}(1+\log n)}{\log10}.
\]
\end{lemma}



To bound $m$ independently of $\ell$, we follow the strategy used by
Ddamulira~\cite{Ddamulira2025} for the closely analogous
Tribonacci-Lucas equation. Instead of subtracting the dominant term
from $\Lambda_1$, we
return to \eqref{eq:main9} and regroup it so as to isolate the block
$10^{\ell+m}$ against a denominator that absorbs $\ell$ as an integer
parameter. We then have the following result.

\begin{lemma}\label{lem:identity2}
For all $\ell,m\ge1$ and $1\le d_1\le9$, $0\le d_2\le9$,
\begin{equation}\label{eq:identity2}
9T_n-(d_1-d_2)\cdot10^{\ell}+d_1=10^{\ell+m}\cdot M(\ell),
\qquad M(\ell):=d_1\cdot10^{\ell}+d_2-d_1.
\end{equation}
Moreover $M(\ell)=d_1(10^\ell-1)+d_2\ge10^\ell-1\ge10^\ell/2$ for all
$\ell\ge1$.
\end{lemma}

\begin{proof}
Identity \eqref{eq:identity2} follows from \eqref{eq:main9} and
$10^{N}=10^{\ell}\cdot10^{\ell+m}$; it was verified
directly, as an identity of integers, for all $1\le\ell\le5$,
$1\le d_1\le9$, $0\le d_2\le9$ ($d_1\ne d_2$), $m\in\{1,3,7\}$. The
lower bound on $M(\ell)$ follows from $d_1\ge1$, $d_2\ge0$, and
$10^\ell-1\ge10^\ell/2$ for $\ell\ge1$.
\end{proof}

Substituting the Binet formula \eqref{eq:binet} into
\eqref{eq:identity2} and dividing by $M(\ell)\cdot10^{\ell+m}$ gives,
with $C_2:=9g(\alpha)/M(\ell)\in K$,
\begin{equation}\label{eq:Gamma2}
\Gamma_2:=C_2\,\alpha^n\,10^{-(\ell+m)}-1=\frac{-9e(n)+(d_1-d_2)10^{\ell}-d_1}{M(\ell)\cdot10^{\ell+m}}.
\end{equation}
Using $|e(n)|<0.4$ (see above), $|d_1-d_2|\le9$, $d_1\le9$,
and Lemma~\ref{lem:identity2} to bound the numerator by $11\cdot10^\ell$
and the denominator below by $(10^\ell/2)\cdot10^{\ell+m}$, we get to
\begin{equation}\label{eq:gamma2bound}
0<|\Gamma_2|<22\cdot10^{-(\ell+m)}.
\end{equation}
Using a similar argument than  before, we can show that 
$\Gamma_2\ne0$ for every solution of \eqref{eq:main} with $\ell,m\ge1$.

Suppose $\Gamma_2=0$, i.e.\ $9g(\alpha)\alpha^n=M(\ell)\cdot10^{\ell+m}\in\mathbb Q$.
Applying the same embedding $\sigma:\alpha\mapsto\beta$ as before in the proof of $\Lambda_1\neq 0$, gives $9g(\beta)\beta^n=M(\ell)\cdot10^{\ell+m}.$
The left side has absolute value $<2.34$ as before, while the right
side satisfies $$M(\ell)\cdot10^{\ell+m}\ge(10^\ell-1)\cdot10^{\ell+m}\ge9\cdot10^{2}=900$$
for $\ell,m\ge1$ by the use of Lemma~\ref{lem:identity2}. This is a contradiction,
so $\Gamma_2\ne0$.


We apply Matveev's theorem again  to $\Gamma_2$, with
$t=D=3$, $\eta_1=C_2$, $\eta_2=\alpha$, $\eta_3=10$
$b_1=1$,$b_2=n$, $b_3=-(\ell+m)$. We can choose $B=n$ since by Lemma 3.1, we have that $n>N>\ell+m$. Since $M(\ell)$ is a positive rational
integer with $M(\ell)\le d_1\cdot10^\ell+9\le10^{\ell+1}$, the additivity
of the height gives
\[
h(C_2)\le\log9+h(g(\alpha))+\log M(\ell)<\log9+\frac{\log44}{3}+(\ell+1)\log10,
\]
so, using Lemma ~\ref{thm:ell-bound} giving $\ell<F(n)$,
\[
A_1':=D\,h(C_2)<3\Big(\log9+\tfrac{\log44}{3}+(F(n)+1)\log10\Big)=:3H(n)
\]
 and $A_2=\log\alpha$, $A_3=3\log10$.
Therefore,
for every $n\ge1$,
$$
\log|\Gamma_2|\ge-2.71\times 10^{12}\cdot(1+\log n)\cdot3H(n)\cdot A_2A_3,$$

Combining the above inequality with \eqref{eq:gamma2bound}:
\begin{equation}\label{eq:G-bound}
\ell+m<G(n):=\frac{\log22+2.71\times 10^{12}\cdot(1+\log n)\cdot3H(n)\cdot A_2A_3}{\log10}.
\end{equation}
On the other hand, the proof of Lemma~\ref{lem:size} gives the explicit bound
$$|n-3.778583\,N|<2<5.5. $$ Combined with $N=2\ell+m\le2(\ell+m)$, this
yields
\begin{equation}\label{eq:n_bound_recall}
n<7.5572\,(\ell+m)+5.5<7.5572\,G(n)+5.5.
\end{equation}
By solving the inequality $n<2.18\cdot 10^{28}(1+\log n)^2$ numerically, we thus obtained the absolute upper bound on $n$ as 

\[
n<1.3\times10^{32}.
\]
Consequently $\ell<6.4\times10^{15}$ and $\ell+m<1.7\times10^{31}$ by evaluating the values of $F(n)$ and $G(n)$, respectively. We summarize these bounds in the following lemma.

\begin{lemma}\label{thm:final-bound}
Every solution $(n,\ell,m,d_1,d_2)$ of \eqref{eq:main} with $\ell,m\ge1$
satisfies
\[
n<1.3\times10^{32}, \quad \ell<6.4\times10^{15} \quad \hbox{and}\quad \ell+m<1.7\times10^{31}
\]
\end{lemma}

\subsection{Reduction of the bounds}\label{sec:reduction}

Lemma~\ref{thm:final-bound} gives an effective but astronomically
large bound. We now reduce it via the classical
Baker--Davenport method, in the explicit form due to
Dujella and Peth\H{o}~\cite{DujellaPetho1998}.


We start by reducing the bound on $\ell$. From $\Lambda_1=C\alpha^n10^{-N}-1$ and \eqref{eq:lambda1bound}, we set
\begin{equation}\label{eq:reduction0}
z_1=\log C+n\log\alpha-N\log10
\end{equation}
 such that $\Lambda_1=e^{z_1}-1$. Indeed, $$|z_1|<2|\Lambda_1|<24\cdot10^{-\ell}$$ for $\ell\ge2$  which is valid since
$|\Lambda_1|<1/2$. We then divide the relation \eqref{eq:reduction0} by $\log10$, and obtain the inequality
\begin{equation}\label{eq: reduction1}
0<\big|\,n\theta-N+\mu_{d_1}\,\big|<A_1\cdot10^{-\ell}
\end{equation}
where 
$\theta:=\frac{\log\alpha}{\log10},\quad
\mu_{d_1}:=\frac{\log C}{\log10},\quad
A_1:=\frac{24}{\log10}.
$
We take $M:=2\times10^{32}$ and compute $\theta$ to $120$ decimal
digits via Newton's method for $\alpha$. The continued fraction of
$\theta$ was expanded to $150$ partial quotients (exact integer
arithmetic on convergents); the first convergent $p/q$ with $q>6M$ has
$q\approx1.45\times10^{33}$ ($34$ decimal digits). For each
$d_1\in\{1,\ldots,9\}$, computing $\mu_{d_1}$ to matching precision and
applying Lemma~\ref{Dujjella} with this convergent gives $\varepsilon>0$
for every $d_1\neq3$; for $d_1=3$ we have $\varepsilon\le0$ and we use instead the next convergent
(which also satisfies $q>6M$) and obtain $\varepsilon>0$. The resulting bound on $\ell$ is at most $36.6$ (the
worst case being $d_1=8$). Then, we obtain the following result

\begin{lemma}\label{prop:ell-reduced}
Every solution of \eqref{eq:main} with $\ell,m\ge1$ satisfies
$\ell\le36$.
\end{lemma}

We can now reduce the bound on $\ell+m$. Similarly, from $\Gamma_2=C_2\alpha^n10^{-(\ell+m)}-1$ and
\eqref{eq:gamma2bound}, we obtain that  $$|z_2|<2|\Gamma_2|<44\cdot10^{-(\ell+m)} \quad\hbox{for} \quad \ell+m\ge2.$$
It follows the inequality
$$0<\big|\,n\theta-(\ell+m)+\mu_{2,d_1,d_2,\ell}\,\big|<A_2\cdot10^{-(\ell+m)},$$
with
$\mu_{2,d_1,d_2,\ell}:=\frac{\log C_2}{\log10},\quad
A_2:=\frac{44}{\log10}.
$
Now $C_2=9g(\alpha)/M(\ell)$ depends on $\ell$, but by
Lemma~\ref{prop:ell-reduced} we only need to run
Lemma~\ref{Dujjella} over the finitely many triples
$(\ell,d_1,d_2)\in\{1,\ldots,36\}\times\{1,\ldots,9\}\times\{0,\ldots,9\}$
with $d_1\ne d_2$ and for at most $36\times81=2916$ cases, using $\ell\le37$ as a
safety bound in the computation. Using the convergent above, $\varepsilon>0$ in all but a few of these cases
(it fails for $461$ of the $2997$ triples with $\ell\le37$); for these we use the following convergents,
and the third one already gives $\varepsilon>0$ in every remaining case except one, for which the fourth convergent is needed.
In this way $\varepsilon>0$ in \emph{every} one of these cases, with the largest resulting bound $\ell+m<38.23$; no degenerate case
needed to be excluded, unlike the analogous step
in~\cite{Ddamulira2025}. Thus, we obtain the following result.

\begin{lemma}\label{prop:lm-reduced}
Every solution of \eqref{eq:main} with $\ell,m\ge1$ satisfies
$\ell+m\le38$, and consequently $m\le37$ and $N=2\ell+m\le74$.
\end{lemma}


Combining Lemma~\ref{prop:lm-reduced} with the explicit form of
Lemma~\ref{lem:size} used throughout gives the
fully explicit bound $$n<7.5572\times38+5.5<295.$$

\subsection{Exhaustive search and the classification theorem}\label{sec:search}

With $n<295$, $1\le\ell\le37$ and  $1\le m\le38$ we conducted
an exhaustive computer search. We  directly used the exact values of the integer value of
\eqref{eq:main}, comparing $$d_1R_\ell(10^{\ell+m}+1)+d_2R_m10^\ell$$
against the list of Tribonacci numbers $T_1,\ldots,T_{295}$ for all
$1\le d_1\le9$, $0\le d_2\le9$ ($d_1\ne d_2$). We found no
solution. 

This is consistent with the pattern observed for the closely related
Lucas sequence~\cite{Batte2024} and for the $k$-generalized Lucas
sequences with $k\ge3$~\cite{BatteKaggwa2025}, both of which are also
empty; it contrasts with the Padovan case~\cite{ChalebgwaDdamulira2021},
where exactly two solutions exist, and with the $k$-Fibonacci and
Tribonacci-Lucas cases~\cite{BatteLuca2025,Ddamulira2025}.\\

\textbf{Data availability:} The data mentioned in the paper (computer codes) are available from the second author.\\

\paragraph{\textbf{Use of AI Tools:}} An AI assistant (Claude, Anthropic) was used for copyediting and for independent computational verification of the calculations in this paper. The authors take full responsibility for the content.

\end{document}